\documentclass[11pt]{article}

\usepackage{amstext,amssymb,amsmath,graphicx,amsthm,ulem,enumitem,tikz-cd,bbm,fancyhdr,microtype,bm,amsfonts,mathrsfs,bbm,xcolor,mathtools}
\usepackage{enumitem}
\usepackage[T1]{fontenc}
\usepackage[margin=1in]{geometry}

\newcommand{\sO}{\mathscr{O}}

\newtheoremstyle{quest}{\topsep}{\topsep}{}{}{\bfseries}{}{ }{\thmname{#1}\thmnote{ #3}.}

\newcommand{\Hom}{\mathrm{Hom}}
\newtheorem{thm}{\textbf{Theorem}}
\newtheorem{cor}[thm]{\textbf{Corollary}}
\newtheorem{lem}[thm]{\textbf{Lemma}}
\newtheorem{pro}[thm]{\textbf{Proposition}}

\newcommand{\comod}{\ensuremath{\mathrm{Comod}}}
\newcommand{\Comod}{\ensuremath{\mathrm{Comod}}}

\newcommand{\op}{\ensuremath{\mathrm{op}}}

\newcommand{\QC}{\ensuremath{\mathrm{QCoh}}}

\newcommand{\DMO}[2]{\DeclareMathOperator{#1}{#2}}
\DMO{\spec}{Spec\text{ }}
\DMO{\Frac}{Frac\text{ }}
\DMO{\cok}{coker}
\DMO{\tr}{tr}
\DMO{\rk}{rk}
\DMO{\img}{im}
\DMO{\kdim}{Krull\text{ }dim}

\renewcommand{\phi}{\varphi}
\DMO{\grad}{\mathrm{grad}}
\DMO{\Spec}{\mathrm{Spec}}
\DMO{\Spet}{\mathrm{Sp\'et}}

\renewcommand{\mod}{\ensuremath{\mathrm{Mod}}}
\newcommand{\Mod}{\ensuremath{\mathrm{Mod}}}

\theoremstyle{definition}
\newtheorem{df}[thm]{\textbf{Definition}}

\newtheorem{exa}[thm]{\textbf{Example}}

\newtheorem{rmk}[thm]{\textbf{Remark}}

\theoremstyle{quest}

\begin{document}
\setlength{\parindent}{15pt}
\setlength{\ULdepth}{1pt}



\title{Algebraic Models for Quasi-Coherent Sheaves in Spectral Algebraic Geometry}

\author{Adam A. Pratt}




\maketitle
\abstract{In this paper we prove the existence of an algebraic model for quasi-coherent sheaves on certain non-connective geometric stacks arising in stable homotopy theory and spectral algebraic geometry using the machinery of adapted homology theories.}

\setcounter{section}{0}
\tableofcontents


\section{Introduction}
\subsection{Summary}
In this paper we prove the existence of an algebraic model for quasi-coherent sheaves on certain non-connective geometric stacks arising in stable homotopy theory and spectral algebraic geometry using the machinery of adapted homology theories. 

In the second section, we give an introduction to the work and necessary background information, beginning with an overview of the problem and a review of the existing literature. In Section \ref{sag} we review background on spectral algebraic geometry, particularly with regard to non-connective geometric spectral stacks. Section \ref{adapt} provides background on adapted homology theories. Section \ref{algthm} uses the machinery described in the previous section to state the algebraicity theorem used in our main results.

The main results of the paper are proved in Section \ref{maincont}. Section \ref{adaptnonconnect} describes the construction of the adapted homology theory used in our model, exploiting the close relationship between non-connective geometric spectral stacks and commutative Hopf algebroids. Section \ref{algmod} then applies the algebraicity theorem to construct the algebraic model, giving explicit conditions on a non-connective geometric spectral stack that imply the existence of such a model.

Finally, Section \ref{exa} describes two prominent examples that apply our theorem. Section \ref{proj} gives conditions for which our theorem applies to the flat projective spaces in spectral algebraic geometry over a base $\mathbb{E}_\infty$-ring. Section \ref{chrom} applies our theorem to certain truncations of the $p$-local moduli stack of oriented formal groups, giving a new algebro-geometric proof of the chromatic algebraicity theorem which gives an algebraic model for chromatic localizations of the category of spectra at large primes.
\subsection{Acknowlegements}
First and foremost, I would like to thank my advisor Brooke Shipley. I owe thanks to many other people for helpful conversations about this thesis, including Greg Taylor, Ethan Devinatz, Irakli Patchkoria, Anish Chedalavada, David Gepner, Jack Burke, and  Maximilien Holmberg-P\'eroux.

\section{Background}

Computations in algebra are amenable to machine computation, while computations in topology remain largely inaccessible. For this reason, algebraic topologists study algebraic \textit{approximations} of categories of topological spaces. The most powerful of these approximations are cohomology theories, which assign to each space algebraic objects that are used to study space. The data of a cohomology theory can be packaged into a single object known as a spectrum. Spectra are the main objects of study in stable homotopy theory: they act more like algebraic objects than spaces do, but they still contain much of the data of the category of spaces. The question remains, however: just how \textit{algebraic} are spectra?

A stable $\infty$-category $\mathcal{C}$ is said to \textit{admit an algebraic model} if there exists an abelian category $\mathcal{A}$ and an equivalence $h\mathcal{C}\simeq hD(\mathcal{A})$ between the homotopy categories of $\mathcal{C}$ and the derived category of $\mathcal{A}$, respectively. An important heuristic in stable homotopy theory, the ``Mahowald Uncertainty Principle" states that any approximation to the stable homotopy groups of spheres (a central object of study in stable homotopy theory) which starts with homological algebra must be infinitely far from the solution \cite{pst1}. More concretely, Schwede proved that the category of spectra does not admit an algebraic model \cite{Schwede}. However, this theme is not common in stable homotopy: stable $\infty$-categories admitting algebraic models abound.

Work of Quillen showed that rational homotopy theory is algebraic \cite{Quillen}. Rational homotopy \textit{localizes} spaces, forcing all algebraic objects used to study spaces to become vector spaces over the rational numbers by requiring that singular homology with rational coefficients detects equivalences. Since this foundational result, there has been significant progress in identifying exactly which stable $\infty$-categories admit algebraic models. Some notable examples include the categories of modules over Eilenberg-Maclane spectra $HR$ for $R$ a commutative ring, 
 and spectra localized with respect to complex topological $K$-theory at odd primes \cite{bousfield}.
 
 The theory of commutative rings in stable homotopy theory, that is, \textit{ring spectra}, has seen significant developments in the past few decades \cite{ekmm,hss,HA}. Much of the classical theory of commutative rings can be generalized to this setting, where commutative rings are replaced with ring spectra that are commutative only up to an infinite chain of higher homotopies, known as \textit{$\mathbb E_\infty$-ring spectra}. It is then natural to ask if we can similarly generalize the theory of algebraic geometry over commutative rings to this setting: this field of study is known as derived algebraic geometry. Derived algebraic geometry studies geometric objects (often called \textit{derived stacks}) whose affine objects are formal duals of objects arising in homological algebra and stable homotopy theory, such as simplicial commutative rings, commutative differential-graded algebras, or $\mathbb{E}_\infty$-ring spectra. It is worth noting that each of these approaches to derived algebraic geometry has its own strengths and weaknesses, and that they are all equivalent for connective objects over a field of characteristic 0. However, the approach based on $\mathbb{E}_\infty$-ring spectra as laid out in \cite{SAG} is the most general approach, and is typically the most useful when the theory is applied to problems in stable homotopy theory. In derived algebraic geometry, many objects of a homological nature typically studied in ordinary algebraic geometry, especially various derived functors, become much more natural. For example, in the derived setting, the global sections functor (defined in a straightforward manner) captures all of the information that would typically be given by sheaf cohomology in the non-derived setting. The theory of derived algebraic geometry has many applications throughout homotopy theory, ordinary algebraic geometry, and arithmetic.
 
This thesis approaches derived algebraic geometry through the lens of algebraic models, examining when certain stable $\infty$-categories arising naturally in derived algebraic geometry, namely the category of quasi-coherent sheaves on a derived stack, admit algebraic models. We prove a theorem that gives specific criteria ensuring such categories admit algebraic models, generalizing a theorem of Patchkoria-Pstragowski in the affine setting \cite{pp21}. We also study specific examples of derived stacks satisfying the hypotheses of the theorem.

\begin{rmk}
    Note that throughout the paper, we will implicitly assume that all categories and categorical constructions (such as functors or limits) are $\infty$-categorical in nature (unless otherwise stated). We follow the framework constructed in \cite{HTT}. 
\end{rmk}

\subsection{Spectral algebraic geometry}\label{sag}
Spectral algebraic geometry is a generalization of classical algebraic geometry in which the spaces (or, more generally, \textit{$\infty$-topoi}) of study are locally modeled on $\mathbb{E}_\infty$-ring spectra (rather than commutative rings in the classical setting). While there are many classifications of such space-like structures, in this work we concern ourselves with a particular class of objects known as nonconnective geometric spectral stacks. In order to define such objects, we need to first recall a few preliminary definitions.
\begin{df}[{Faithful flatness, \cite[Definition 2.39]{Ma18}}]
    Recall that a map of $\mathbb{E}_\infty$-ring spectra $A\to B$ is called \textit{faithfully flat} if the induced map $\pi_0A\to \pi_0B$ is a faithfully flat map of commutative rings and the map \[\pi_*A\otimes_{\pi_0A}\pi_0B\to \pi_*B\] is an isomorphism.
\end{df}
\begin{rmk}
    The above definition allows us to define the \textit{fpqc} topology on the category of $\mathbb{E}_\infty$-ring spectra $\mathrm{CAlg}$, where covers are given by faithfully flat maps \cite[B.6.1.3]{SAG}.
\end{rmk}

We can now define our main geometric objects of study:

\begin{df}\label{ncgeo} Following \cite[Definition 1.3.1]{rok1}, we define a \textit{nonconnective geometric spectral stack} as a functor $\mathfrak X:\mathrm{CAlg}\to \mathcal S$ from the category of $\mathbb{E}_\infty$-ring spectra to the category of spaces satisfying:
    \begin{enumerate}
        \item The functor $\mathfrak X$ satisfies descent for the fpqc topology.
        \item The diagonal map $\mathfrak X\xrightarrow{\Delta} \mathfrak X\times \mathfrak X$ is affine.
        \item There exists a faithfully flat affine cover $\Spec(A)\to \mathfrak X$ for some $A\in\mathrm{CAlg}$.
    \end{enumerate}
\end{df}
\begin{rmk}
    Note that the affine cover $\Spec(A)\to \mathfrak X$ for a given nonconnective geometric spectral stack $\mathfrak X$ can be extended via the \v{C}ech complex to a simplicial presentation (i.e. the geometric realization of a groupoid object) \[|\Spec(A^\bullet)|\simeq \mathfrak X;\] so, since every nonconnective geometric spectral stack admits such a cover, they also admit such a simplicial presentation \cite[1.3.5]{rok1}.
\end{rmk}
\begin{df}[Quasi-coherent sheaves]
    The functor \[\mathrm{QCoh}:(\mathrm{Shv}^\mathrm{nc}_\mathrm{fpqc})^\op\to \mathrm{CAlg}(\mathrm{Pr}^L)\] from the opposite category of nonconnective fpqc sheaves to the category of presentably symmetric monoidal infinity categories is defined by right Kan extension along the map $A\mapsto \mod_A$; that is, given a simplicial presentation $\varinjlim_i\Spec(A_i)\simeq \mathfrak X$, we have \[\mathrm{QCoh}(\mathfrak X)\simeq \QC(\varinjlim_i\Spec(A_i)) \simeq \varprojlim_i\mod_{A_i}.\] The resulting symmetric monoidal $\infty$-category $\QC(\mathfrak{X})$ is known as the category of \textit{quasi-coherent sheaves} on $\mathfrak{X}$ \cite[1.4.1]{rok1}.
\end{df}

\begin{exa}
    Many classical stacks appearing in stable homotopy have spectral analogues in the form of nonconnective geometric spectral stacks, such as the moduli stack of formal groups \cite{rok1} and the moduli stack of elliptic curves (this is the content of the Goerss-Hopkins-Miller-Lurie theorem) \cite{lurie}.
\end{exa}

\subsection{Adapted homology theories}\label{adapt}

In this section we describe the functors which allow us to relate our stable $\infty$-categories to certain abelian categories, giving us a starting point for the construction of our algebraic models. Such functors are known as \textit{homology theories}. We will start with a slightly more general definition.

\begin{rmk}
    While many of the definitions in this section are classical in nature, our main source for this material is \cite{pp21}, as that work synthsizes all of these classical work (along with new theory) to lay the modern foundations for the theory of algebraic models we will employ.
\end{rmk}

\begin{df} A functor $H$ from a stable $\infty$-category $\mathcal C$ to an abelian category $\mathcal A$ is called \textit{homological} if it is additive and takes cofiber sequences in $\mathcal C$ to exact sequences in $\mathcal A$.
\end{df}

To upgrade a homological functor to a homology theory as needed, we will need additional structure on the source and target of the functor, which we will now describe.

\begin{df} Given an $\infty$-category $\mathcal C$, a \textit{local grading} on $\mathcal C$ is an auto-equivalence \[[1]:\mathcal C\to \mathcal C.\] A pair $(\mathcal C, [1])$ of an $\infty$-category $\mathcal C$ with a choice of local grading $[1]$ is then called a \textit{locally graded $\infty$-category}.

A functor $F:\mathcal C\to \mathcal D$ between two locally graded $\infty$-categories is a \textit{functor of locally graded $\infty$-categories} if it respects the local grading; that is, we have an equivalence \[F\left(X[1]_\mathcal{C}\right)\simeq F(X)[1]_\mathcal{D}\]for all $X\in \mathcal C$.
\end{df}

\begin{exa}
    Every stable $\infty$-category $\mathcal C$ has a local grading given by the suspension functor \[\Sigma:\mathcal C\to\mathcal C.\] For the remainder of this thesis, when working with a stable $\infty$-category, we will implicitly assume it is locally graded in this manner.
\end{exa}

\begin{df}
A functor $H:\mathcal C\to A$ of locally graded $\infty$-categories is called a \textit{homology theory} if its underlying functor is homological.
\end{df}

In order for our homology theories to be useful in constructing algebraic models, they need to interact nicely with injective resolutions in the target, allowing us to lift such resolutions to an analogous structure in the source.

\begin{df}For $H:\mathcal C\to \mathcal A$ a homology theory and $i\in \mathcal A$ an injective object, we say an object $i_\mathcal{C}\in \mathcal C$ representing the functor \[ \Hom_\mathcal{A}(H(-),i):\mathcal C^\op \to \mathsf{Ab} \] in the homotopy category $h\mathcal C$ is an \textit{injective lift for $i$}.
\end{df}

From representability, we obtain an isomorphism \[[i_\mathcal{C},i_\mathcal{C}]\xrightarrow{\cong} \Hom_\mathcal{A}(H(i_\mathcal{C}),i);\] the image of the identity under this map supplies us with a map 
\begin{equation}\label{tau}
    H(i_\mathcal{C})\to i.
\end{equation}

The following condition on a homology theory will ensure the existence of an Adams spectral sequence on the source, which is a necessary ingredient in the construction of an algebraic model by way of the algebraicity theorem.

\begin{df} A homology theory $H:\mathcal C\to \mathcal A$ is \textit{adapted} if
\begin{enumerate}
\item $\mathcal A$ has enough injectives,
\item Any injective $i\in\mathcal A$ admits an injective lift $i_\mathcal{C}\in\mathcal C$, and
\item The structure morphism (\ref{tau}) is an isomorphism for any $i$.
\end{enumerate}
\end{df}

\begin{exa}\label{adex}
    For $R$ an $\mathbb{E}_1$-ring spectrum, the homotopy functor \[\pi_*:\mod_R\to\mod_{R_*}\] is an adapted homology theory \cite[6.53]{pp21}.
\end{exa}

\begin{exa}\label{chex}
    Let $R$ be an \textit{Adams-type} ring spectrum, as defined in \cite{pst1}. Then the comodule-valued $R$-homology functor \[R_*:\mathrm{Sp}\to \Comod_{R_*R}\] is an adapted homology theory as a consequence of a theorem of Devinatz \cite[1.5]{dev}.
\end{exa}

\subsection{The algebraicity theorem}\label{algthm}
In this section, we introduce the algebraicity theorem (Theorem \ref{franke}, \cite{pp21}) that we will use in the construction of our algebraic model. We will begin by recalling the definition of the periodic derived category of an abelian category; such $\infty$-categories will be the target of the algebraic models constructed. All of the definitions in this section are classical in nature; however, our main reference is \cite{pp21}.

\begin{df}
    Let $\mathcal A$ be a locally graded abelian category. A \textit{differential object} in $\mathcal A$ is a pair $(M, d)$, where $M\in \mathcal A$ and $d:M\to M$ is a differential; that is, $d[1]\circ d = 0$. A morphism of differential objects is a morphism in $\mathcal A$ that commutes with the differentials.
\end{df}

\begin{df}
If $(M,d)$ is a differential object, then its \textit{homology} is defined to be \[H(M,d):=\ker(d)/\img(d[-1]).\]
A map of differential objects is said to be a \textit{quasi-isomorphism} if it induces an isomorphism on homology.
\end{df}
The category of differential objects is abelian, and we can endow it with a model structure where the weak equivalences are the quasi-isomorphisms and the cofibrations are the monomorphisms (provided $\mathcal A$ has enough injectives and is of finite cohomological dimension, which will always be the case for us).
\begin{df}
    The \textit{periodic derived category} $ D^{\mathrm{per}}(\mathcal A)$ of $\mathcal A$ is the $\infty$-categorical localization of the above model structure on differential objects at the class of quasi-isomorphisms.
\end{df}

\begin{rmk}
    It is worth noting that much of the literature on algebraic models (including Franke's original paper) is written in terms of \textit{twisted chain complexes}, that is, chain complexes admitting a certain periodicity isomorphism between shifts of the complex. The category of such objects admits a model structure Quillen equivalent to the one described above for differential objects; we choose to follow the convention of \cite{pp21} in the use of differential objects (rather than twisted chain complexes) in order to more easily apply their presentation of the algebraicity theorem.
\end{rmk}

We now recall a few more definitions which are necessary in the statement of the theorem and will be used again throughout this work.

\begin{df}
A locally graded abelian category $\mathcal A$ is said to \textit{admit a splitting of order} $n$ if it can be decomposed as a product $\prod_\phi \mathcal A_\phi$ of Serre subcategories $\mathcal A_\phi$ indexed by $\phi\in \mathbb{Z}/n\mathbb{Z}$ such that $[1]\mathcal A_\phi\subseteq \mathcal A_{\phi+1}$.
\end{df}

\begin{df}
We say that an abelian category $\mathcal A$ has \textit{finite cohomological dimension} $d$ if $d$ is the smallest natural number such that for all $X,Y\in \mathcal A$ and $n>d$, the groups $\mathrm{Ext}^n_{\mathcal A}(X,Y)$ vanish.
\end{df}

We now recall the statement of the algebraicity theorem, as proved in \cite{pp21}:

\begin{thm}[Franke's algebraicity conjecture, \cite{pp21}]\label{franke} Suppose that $H:\mathcal C \to \mathcal A$ is a conservative, adapted homology theory, and assume that \begin{enumerate}
    \item $\mathcal A$ admits a splitting of order $q+1$,
    \item $\mathcal A$ is of finite cohomological dimension $d$, and
    \item $q\geq d$.
\end{enumerate}
Then, we have an equivalence of homotopy $(q+1-d)$-categories \[ h_{q+1-d}D^\mathrm{per}(\mathcal A) \simeq h_{q+1-d}\mathcal C. \]
\end{thm}

Our main theorem (Theorem \ref{main}) generalizes the following result in the affine setting, originally proved by Pathckoria in \cite{patch} (strengthened in \cite{pp21}), which applies the algebraicity theorem to the adapted homology theory $\pi_*:\mod_R\to\mod_{R_*}$ on $R$-module spectra from Example \ref{adex}:

\begin{exa}[Theorem 8.2, \cite{pp21}] \label{affine}
Let $R$ be an $\mathbb{E}_\infty$-algebra in spectra and suppose that:
\begin{enumerate}
    \item $R_*$ is concentrated in degrees divisible by $(q+1)$ and
    \item $\Mod_{R_*}$ is of cohomological dimension $d\leq q$.
\end{enumerate}
Then, there exists an equivalence \[h_{q+1-d}\Mod_R\simeq h_{q+1-d}D^\mathrm{per}(\Mod_{R_*})\] between the homotopy $(q+1-d)$-categories of $R$-module spectra and the periodic derived category of $R_*$-modules.
\end{exa}
Note that the full result in \cite{pp21} only requires the ring spectrum in question to be $\mathbb{E}_1$; due to the algebro-geometric nature of the present work, our theorem only generalizes the $\mathbb{E}_\infty$ case.
\color{black}
\section{Algebraic models in spectral algebraic geometry}\label{maincont}


\subsection{Adapted homology theories and nonconnective geometric spectral stacks}\label{adaptnonconnect}

Throughout this section, fix a nonconnective geometric spectral stack $\mathfrak X$ with structure sheaf $\sO$ and faithfully flat cover $f:\Spec A\to \mathfrak X$.

We will begin by identifying the category $\mathrm{QCoh}(\mathfrak X)$ of quasi-coherent sheaves on $\mathfrak X$ with the category of comodules over a certain coalgebra in spectra. Note that, since the diagonal embedding for $\mathfrak X$ is affine, we have that the pullback $\Spec A \times_{\mathfrak X} \Spec A$ is also affine; denote its commutative ring spectrum of global sections by $\Gamma$. $\Gamma$ also has the structure of an $\mathbb{E}_1$-coalgebra in $(A,A)$-bimodules as described in \cite{torii}. It is easy to see that $\Gamma_*$ is a commutative Hopf algebroid over $A_*$.

As in \cite[2.4.8]{rok1}, we have the  following lemma:
\begin{lem}\label{equiv}There is an equivalence of $\infty$-categories
    \[\mathrm{QCoh}(\mathfrak X) \simeq \mathrm{Comod}_{\Gamma}\] between the category of quasi-coherent sheaves on $\mathfrak X$ and the category of comodules over the coalgebra $\Gamma$. Under this equivalence, we can identify our inverse-direct image adjunction \[\QC(\mathfrak X)\rightleftarrows \mathrm{Mod}_A\] on quasi-coherent sheaves with the forgetful-cofree adjunction \[\mathrm{Comod}_{\Gamma}\rightleftarrows \mathrm{Mod}_A\] on comodules.
\end{lem}
\begin{proof}
By the theory of comonadic descent \cite[4.7.5.3]{HA}, we see that we have equivalences 
\[\QC(\mathfrak X)\xrightarrow{\sim}\mathrm{Tot}(\Mod_{A_i}) \]
and
\[\begin{tikzcd}[column sep=small,ampersand replacement=\&]
	{\mathrm{Comod}_\Gamma} \& {\mathrm{Tot}\Big(\Mod_A} \& {\Mod_\Gamma} \& {\Mod_{\Gamma \otimes_A \Gamma}} \& {\cdots\Big)}\& {\mathrm{Tot}\Big(\Mod_{\Gamma^{\otimes_A \bullet}}\Big)}
	\arrow["\sim", from=1-1, to=1-2]
	\arrow[shift right=.75, from=1-2, to=1-3]
	\arrow[shift left=.75, from=1-2, to=1-3]
	\arrow[shift right=1.5, from=1-3, to=1-4]
	\arrow[shift left=1.5, from=1-3, to=1-4]
	\arrow[from=1-3, to=1-4]
	\arrow[shift left=2.25, from=1-4, to=1-5]
	\arrow[shift right=2.25, from=1-4, to=1-5]
	\arrow[shift left=.75, from=1-4, to=1-5]
	\arrow[shift right=.75, from=1-4, to=1-5]
 \arrow["="{description}, draw=none, from=1-5, to=1-6]
\end{tikzcd}\] where $\{A_i\}$ is the cosimplicial $\mathbb{E}_\infty$ ring given by the global sections of the \v{C}ech nerve of the cover $\Spec{A}\to \mathfrak X$. Since $\Gamma$ is defined to be the global sections of $\Spec{A}\times_{\mathfrak X}\Spec{A}$, the right hand sides are equivalent, as \begin{align*}
    \Gamma\otimes_A\Gamma &\simeq \mathscr{O}(\Spec{A}\times_{\mathfrak X}\Spec{A})\otimes_A\mathscr{O}(\Spec{A}\times_{\mathfrak X}\Spec{A})\\
    &\simeq \mathscr{O}(\Spec{A}\times_{\mathfrak X}\Spec{A})\otimes_{\mathscr{O}(\Spec A)}\mathscr{O}(\Spec{A}\times_{\mathfrak X}\Spec{A})\\
    &\simeq \mathscr{O}(\Spec{A}\times_{\mathfrak X}\Spec{A}\times_{\Spec A}\Spec{A}\times_{\mathfrak X}\Spec{A})\\
    &\simeq
\mathscr{O}(\Spec{A}\times_{\mathfrak X}\Spec{A}\times_{\mathfrak X}\Spec{A});
\end{align*} therefore, the left hand sides are also equivalent. Moreover, this equivalence sends the adjunctions $(f^* \dashv f_*)$ and $(U \dashv \Gamma \otimes_A -)$ to each other as two comonadic adjunctions presenting the same comonad.
\end{proof}

Since $\Gamma$ is flat over $A$, the K\"unneth spectral sequence calculating $\pi_*(\Gamma\otimes_A N)$ collapses for any $N\in \mathrm{Comod}_{\Gamma}$, giving us an isomorphism \[\pi_*(\Gamma\otimes_A N)\cong \Gamma_*\otimes_{A_*} \pi_*N\] coming from the edge homomorphism. So, $\pi_*$ preserves the comodule structure on $N$; that is, $\pi_*N$ is a $\Gamma_*$-comodule, and the homotopy functor on $\Gamma$-comodules factors through the category of $\Gamma_*$-comodules: \[\pi_*:\mathrm{Comod}_{\Gamma}\to \mathrm{Comod}_{\Gamma_*}.\] This functor will be the adapted homology theory we will use in the construction of our algebraic model. 

\begin{thm}\label{homtheory}
The functor $\pi_*:\mathrm{Comod}_{\Gamma}\to \mathrm{Comod}_{\Gamma_*}$ is a homology theory.
\end{thm}

\begin{proof}

We first need to show that the homotopy functor on comodules is a homological functor; that is, it preserves arbitrary sums and takes cofiber sequences to short exact sequences. Note that since the homotopy groups of a $\Gamma$-comodule are determined by the underlying $A$-module, we have a commutative diagram 
\begin{equation}\label{pi}
\begin{tikzcd}[row sep=large]
	{\comod_\Gamma} & {\mod_A} \\
	{\comod_{\Gamma_*}} & {\mod_{A_*}.}
	\arrow["U", from=1-1, to=1-2]
	\arrow["{\pi_*}", from=1-2, to=2-2]
	\arrow["{\pi_*}"', from=1-1, to=2-1]
	\arrow["{\overline{U}}"', from=2-1, to=2-2]
\end{tikzcd}
\end{equation}
To see that $\pi_*$ preserves arbitrary sums of comodules, that is, \[\pi_*(\bigoplus_i N_i)\cong \bigoplus_i\pi_*(N_i)\] for all $N_i\in \comod_\Gamma$, note that \begin{align*}
\overline{U}(\pi_*(\bigoplus_i N_i))&\cong \pi_*(U(\bigoplus_i N_i))\\
&\cong \pi_*(\bigoplus_i U(N_i))\\
&\cong \bigoplus_i\pi_*( U(N_i))\\
&\cong \bigoplus_i\overline{U}(\pi_*(N_i))\\
&\cong\overline{U}(\bigoplus_i\pi_*(N_i)).
\end{align*}
Here the first isomorphism is (\ref{pi}), the second comes from the fact that $U$ is a left adjoint, the third uses the fact that $\pi_*$ is a homological functor on modules and therefore presrves sums, the fourth is (\ref{pi}), and the final isomorphism uses the fact that $\overline{U}$ is a left adjoint. Finally, since $\overline U$ is conservative, we see that $\pi_*(\bigoplus_i N_i)\cong \bigoplus_i\pi_*(N_i)$ as needed, meaning $\pi_*$ commutes with arbitrary sums of comodules. 

To see that $\pi_*$ takes cofiber sequences in $\comod_\Gamma$ to exact sequences in $\comod_{\Gamma_*}$, let $N_1\to N_2\to N_3$ be a cofiber sequence in $\comod_\Gamma$. Then $U(N_1)\to U(N_2)\to U(N_3)$ is a cofiber sequence in $\mod_A$ as $U$ is a left adjoint, meaning $\pi_*U(N_1)\to \pi_*U(N_2)\to \pi_*U(N_3)$ is an exact sequence in $\mod_{A_*}$, as $\pi_*$ is a homological functor on modules. An application of (\ref{pi}) then tells us that $\overline{U}(\pi_*(N_1))\to \overline{U}(\pi_*(N_2))\to \overline{U}(\pi_*(N_3))$ is an exact sequence in $\mod_{A_*}$. Finally, since $\overline U$ is a conservative left adjoint, it reflects any colimits that exist in $\comod_{\Gamma_*}$, meaning $\pi_*(N_1)\to \pi_*(N_2)\to \pi_*(N_3)$ is an exact sequence in $\comod_{\Gamma_*}$, as needed. So, since $\pi_*:\mathrm{Comod}_{\Gamma}\to \mathrm{Comod}_{\Gamma_*}$ preserves arbitrary sums and takes cofiber sequences to short exact sequences, it is a homological functor.

To see that $\pi_*$ is a \textit{homology theory} on $\Gamma$-comodules, we need to see that it preserves the local grading, i.e. it takes suspensions of comodule spectra to shifts of homotopy comodules. Since suspensions in $\comod_\Gamma$ are given by suspension of the underlying spectrum, we see that suspension commutes with the forgetful functor $U$, meaning for any $\Gamma$-comodule $N$, we have \[\pi_*(\Sigma N)\cong( \pi_*N)[1]\] as needed. So, $\pi_*:\mathrm{Comod}_{\Gamma}\to \mathrm{Comod}_{\Gamma_*}$ is a homology theory.
\end{proof}

The hypotheses of the algebraicity theorem additionally require that our homology theory be \textit{adapted}. In order to show this, we will need to recall the following presentation of Brown's representability theorem from \cite[2.15]{pp21}:

\begin{pro}[Brown representability]
    Let $\mathcal C$ be a presentable stable $\infty$-category and let $E : \mathcal C^\op \to \mathsf{Ab}$ be a homological functor which takes arbitrary direct sums in $\mathcal C$ to products of abelian groups. Then, $E$ is representable in the homotopy category $h\mathcal C$.
\end{pro}

We can now prove that our homology theory is adapted:
\color{black}
\begin{thm}\label{adapted}
The homology theory $\pi_*:\mathrm{Comod}_{\Gamma}\to \mathrm{Comod}_{\Gamma_*}$ is adapted.
\end{thm}
\begin{proof}
In order to see that our homology theory is adapted, let $i\in \mathrm{Comod}_{\Gamma_*}$ be an injective comodule. It is easy to see that \[\mathrm{Hom}_{\Gamma_*}(\pi_*(-), i):\mathrm{Comod}_\Gamma^{\op}\to \mathrm{Ab}\] satisfies the conditions of Brown representability, meaning there exists an $I\in \mathrm{Comod}_{\Gamma}$ and an isomorphism \[\pi_*\mathrm{Map}_{\Gamma}(M, I) \cong \mathrm{Hom}_{\Gamma_*}(\pi_*(M), i)\] for all $M\in \mathrm{Comod}_{\Gamma}$. We can see by Corollary 2.17 in \cite{pp21} that $\pi_*$ has lifts for injectives, and taking $M=A$ shows that the structure morphism $\pi_*I\to i$ is an isomorphism, meaning $\pi_*:\mathrm{Comod}_{\Gamma}\to \mathrm{Comod}_{\Gamma_*}$ is adapted.

\end{proof}

\begin{cor}
There is an adapted homology theory \[\pi_*:\QC(\mathfrak X)\to \Comod_{\Gamma_*}.\]
\end{cor}
\begin{proof}
    This is immediate from Theorem \ref{adapted}, by precomposing the homology theory of Theorem \ref{homtheory} with the equivalence of Lemma \ref{equiv}.
\end{proof}

The homology theory can also be defined as the homotopy sheaves defined via pullback as in \cite{rok1} for a more algebro-geometric interpretation. We choose to work with comodules in this thesis due in part to their ubiquity in the literature surrounding algebraic models.


\subsection{The algebraic model}\label{algmod}

In this section, we prove that our situtation satisfies the remaining hypotheses of the algebraicity theorem, namely that our adapted homology theory is conservative and its target admits a certain splitting. We then apply the theorem to construct our algebraic model. 

\begin{lem}
The functor $\pi_*:\mathrm{Comod}_{\Gamma}\to \mathrm{Comod}_{\Gamma_*}$ is conservative.
\end{lem}

\begin{proof}

It is clear that $\pi_*$ is conservative on modules, as homotopy groups detect equivalences of $A$-modules. With notation as in \ref{pi}, both $U$ and $\overline{U}$ are comonadic, and therefore also conservative. Let $N_1,N_2\in \comod_\Gamma$ with $\pi_*N_1\cong \pi_*N_2$. Then $\overline U(\pi_*N_1)\cong \overline U(\pi_*N_2)$, meaning $\pi_*U(N_1)\cong \pi_*U(N_2)$. So, since $\pi_*$ is conservative on $\mod_A$ and $U$ is also conservative, we have $N_1\simeq N_2$. So,  $\pi_*:\mathrm{Comod}_{\Gamma}\to \mathrm{Comod}_{\Gamma_*}$ is conservative.
\end{proof}
\begin{cor}\label{consadap}
    The adapted homology theory $\pi_*:\QC(\mathfrak X)\to \Comod_{\Gamma_*}$ is conservative.
\end{cor}

The hypotheses of our main theorem will require that the target of our homology theory splits as a certain product. We will need the following lemma in the proof of our main theorem:

\begin{lem}
Let $\Gamma$ be as above and assume $\Gamma_*$ is concentrated in degrees divisible by $q+1$. Then $\mathrm{Comod}_{\Gamma_*}$ admits a splitting of order $q+1$.
\end{lem}

\begin{proof}
The proof of this statement is the same as the proof of \cite[Lemma 8.1]{pp21}.
\end{proof}

\begin{thm}\label{main}
Let $\mathfrak X$ be a nonconnective geometric spectral stack with faithfully flat cover $\Spec(A)\to \mathfrak X$ and corresponding coalgebra $\Gamma$. Suppose that the Hopf algebroid $(A_*,\Gamma_*)$ satisfies the following:
\begin{enumerate}
\item $\Gamma_*$ is concentrated in degrees divisible by $q+1$ and
\item the category $\mathrm{Comod}_{\Gamma_*}$ of $\Gamma_*$-comodules is of cohomological dimension $d\leq q$.
\end{enumerate}

Then there exists a canonical equivalence
\[h_{q+1-d}\QC(\mathfrak X)\simeq h_{q+1-d} D^\mathrm{per}(\mathrm{Comod}_{\Gamma_*})\] between the homotopy $(q+1-d)$-categories of quasi-coherent sheaves on $\mathfrak X$ and the periodic derived category of $\Gamma_*$-comodules.
\end{thm}

\begin{proof}
This is immediate from Corollary \ref{consadap}, the previous lemma, and an application of Franke's algebraicity conjecture, proved in  \cite[7.5.6]{pp21}.
\end{proof}

\begin{rmk}[Affine case] Taking $\mathfrak X$ to be affine, that is, $\mathfrak X\simeq \Spec A$, we see that the Hopf algebroid $(A_*,\Gamma_*)$ is the trivial Hopf algebroid $(A_*,A_*)$, giving us $\Comod_{\Gamma_*}\simeq \Mod_{A_*}$; Theorem \ref{main} then reduces to the case of Example \ref{affine}, showing that the former is indeed a generalization of the latter.
\end{rmk}

\section{Examples}\label{exa}

In this chapter, we apply Theorem \ref{main} to give a number of examples of nonconnective geometric spectral stacks for which the category of quasi-coherent sheaves admits an algebraic model.

\subsection{Projective space}\label{proj}

For our first example, we consider algebraic models for sheaves on a spectral analogue of projective spaces from classical algebraic geometry. These spectral projective spaces are constructed as spectral Deligne-Mumford stacks, a different type of stack than the nonconnective geometric spectral stacks considered in the rest of the paper. While neither type of stack is a subcategory of the other, we will show that in our case, both definitions are satisfied. We begin by recalling the definition of spectral Deligne-Mumford stacks \cite[Definition 1.4.4.2]{SAG}, along with the construction of these spectral projective spaces.

\begin{df} A \textit{nonconnective spectral Deligne-Mumford stack} is a spectrally ringed $\infty$-topos \cite[Definition 1.4.1.1]{SAG} $(\mathfrak{X},\sO_\mathfrak{X})$ for which there exists a collection of objects $U_\alpha \in \mathfrak{X}$ satisfying the following conditions:
\begin{enumerate}
    \item The objects $U_\alpha$ cover $\mathfrak{X}$; that is, the map \[\coprod_\alpha U_\alpha \to \mathbf{1}_{\mathfrak{X}}\] is an effective epimorphism, where $\mathbf{1}_{\mathfrak{X}}$ denotes the final object of $\mathfrak X$.
    \item For each index $\alpha$, there exists an $\mathbb{E}_\infty$-ring $A_\alpha$ and an equivalence of spectrally ringed $\infty$-topoi \[\left(\mathfrak{X}_{/U_\alpha},\sO_\mathfrak{X}\big{|}_{U_\alpha}\right)\simeq \Spet A_\alpha.\]
\end{enumerate}

Here $\Spet$ refers to the right adjoint to the global sections functor \[ \Gamma:\infty\mathrm{Top}^\mathrm{sHen}_\mathrm{CAlg}\to \mathrm{CAlg}^\op \] on the category of strictly Henselian spectrally ringed $\infty$-topoi \cite[Definition 1.4.2.1]{SAG}, an \'etale site-analog of the usual $\Spec$ functor \cite[Proposition 1.4.2.3]{SAG}.

We let $\mathrm{SpDM^{nc}}$ denote the full subcategory of $\infty\mathrm{Top}^\mathrm{sHen}_\mathrm{CAlg}$ spanned by the nonconnective
spectral Deligne-Mumford stacks.
A \textit{spectral Deligne-Mumford stack} is a nonconnective spectral Deligne-Mumford stack $(\mathfrak{X},\sO_\mathfrak{X})$ for which the structure sheaf $\sO_\mathfrak{X}$ is connective. We let $\mathrm{SpDM}$ denote the full subcategory of $\mathrm{SpDM^{nc}}$ spanned by the spectral Deligne-Mumford stacks.
\end{df}

It is worth noting that the definition above varies significantly from our definition of nonconnective geometric spectral stacks, which is based on the \textit{functor of points} approach to (spectral) algebraic geometry; this definition relies on the spectral analog of the ``ringed spaces" approach to algebraic geometry, i.e. that of \textit{spectrally ringed $\infty$-topoi}. 

Recall the following definition which will be used in our construction of spectral projective spaces below.

\begin{df}
For a discrete commutative monoid $X$ and an $\mathbb{E}_{\infty}$-ring $R$, denote by $R[X]:= R\otimes \Sigma_+^\infty X$ the \textit{monoid algebra of $X$ over $R$} as defined in \cite[5.4.1.1]{SAG}. Note also that $R[X]$ is flat as an $R$-algebra, giving us an isomorphism \[\pi_*R[X]\cong \pi_0R[X]\otimes_{R_0}R_*,\] and note that $\pi_0R[X]$ can be identified with the classical monoid algebra $(\pi_0R)[X]$.\end{df}

We recall the following definition from \cite[5.4.1.3]{SAG}:

\begin{df}[Flat projective space] Let $[n]$ denote the set $\{0<1<\cdots<n\}$, and let $P^\circ([n])\subset P([n])$ denote the subset of the power set of $[n]$ given by all of the non-empty subsets of $[n]$. For each subset $I\subseteq [n]$, let $M_I$ denote the subset of $\mathbb{Z}^{n+1}$ consisting of all tuples $(k_0,\dots, k_n)$ satisfying $k_0+\cdots+k_n=0$ and $k_i\geq 0$ for $i\notin I$. Then $M_I$ is a commutative monoid, and the assignment $I\mapsto M_I$ is functorial. As above, for $R$ a connective $\mathbb{E}_\infty$-ring, let $R[M_I]$ denote the monoid algebra associated to $M_I$ over $R$. Note that the association $I\mapsto R[M_I]$ is a functor $P([n])\to \mathrm{CAlg}^\mathrm{cn}_R$ from the power set of $[n]$ to the category of connective, commutative $R$-algebras. Composing with the \'etale spectrum functor, we get a functor valued in the category of spectral Deligne-Mumford stacks given by $I\mapsto \Spet R[M_I]$. We define the \textit{projective space of dimension n over R}, $\mathbb{P}^n_R$, as the colimit \[ \mathbb{P}^n_R := \varinjlim_{I\in P^\circ([n])^{\op}} \Spet R[M_I],\] formed in the category $\mathrm{SpDM}$ of spectral Deligne-Mumford stacks as defined above. 
\end{df}

It is worth noting that there are two non-equivalent definitions of projective spaces in spectral algebraic geometry: the flat projective spaces defined above, and smooth projective space, which we will not consider in the present work. 

We will now show that the flat projective spaces defined above are also nonconnective geometric spectral stacks.

\begin{lem}
    Let $R$ be a connective $\mathbb{E}_\infty$-ring, and let $\mathbb{P}^n_R$ be the flat projective space of dimension $n$ over $R$ as defined above. Then $\mathbb{P}^n_R$ satisfies the conditions of Definition \ref{ncgeo}; that is, $\mathbb{P}^n_R$ is a nonconnective geometric spectral stack.
\end{lem}

\begin{proof}
    In order to show that $\mathbb{P}^n_R$ is a \textit{geometric stack} in the sense of \cite[Definition 9.3.0.1]{SAG}, we need to show by \cite[Remark 9.3.0.2]{SAG} that it is \textit{quasi-geometric} and the diagonal $\delta:\mathbb{P}^n_R\to \mathbb{P}^n_R\times \mathbb{P}^n_R$ is affine. 
    
    By \cite[Proposition 5.4.1.7]{SAG}, we know that $\mathbb{P}^n_R$ is a spectral algebraic space, and it admits an open cover by a finite number of affines \begin{equation}\coprod_{0\leq i\leq n}\Spet R[M_{\{i\}}]\to \mathbb{P}^n_{R}.\label{cover}\end{equation} In order to show that $\mathbb{P}^n_{R}$ is quasi-geometric, it will suffice by \cite[Corollary 9.1.4.6]{SAG} to show that it is quasi-compact and quasi-separated. Since each affine is quasi-compact, we have an open cover by finitely many quasi-compact objects, meaning $\mathbb{P}^n_{R}$ is quasi-compact. 
    
    To see that the diagonal is affine, note that the fiber products \[ \Spet R[M_{\{i\}}]\times_{\mathbb{P}^n_{R}}\Spet R[M_{\{j\}}] \simeq \Spet R[M_{\{i,j\}}]  \] of the open cover \ref{cover} are affine. The cover (\ref{cover}) gives us an open cover \begin{equation} \coprod_{i,j}\Spet R[M_{\{i\}}]\times\Spet R[M_{\{j\}}]\to \mathbb{P}^n_{R}\times \mathbb{P}^n_{R}\label{cover2}\end{equation} of the diagonal by a finite number of affines. We have the following pullback square:
\[\begin{tikzcd}
	{\mathbb{P}^n_{R}\times_{\mathbb{P}^n_{R}\times \mathbb{P}^n_{R}}(\Spet R[M_{\{i\}}]\times\Spet R[M_{\{j\}}])} & {\Spet R[M_{\{i\}}]\times\Spet R[M_{\{j\}}]} \\
	{\mathbb{P}^n_{R}} & {\mathbb{P}^n_{R}\times\mathbb{P}^n_{R}}
	\arrow[from=1-1, to=1-2]
	\arrow[from=1-1, to=2-1]
	\arrow[from=1-2, to=2-2]
	\arrow[from=2-1, to=2-2]
\end{tikzcd}\]
We have \begin{align*}
\mathbb{P}^n_{R}\times_{\mathbb{P}^n_{R}\times \mathbb{P}^n_{R}}(\Spet R[M_{\{i\}}]\times\Spet R[M_{\{j\}}])&\simeq \Spet R[M_{\{i\}}]\times_{\mathbb{P}^n_{R}}\Spet R[M_{\{j\}}] \\
&\simeq \Spet R[M_{\{i,j\}}], 
\end{align*} meaning the affines in the cover (\ref{cover2}) pull back to affines. So, the diagonal map $\delta:\mathbb{P}^n_{R}\to \mathbb{P}^n_{R}\times \mathbb{P}^n_{R}$ is affine.

    Finally, since the diagonal $ \delta:\mathbb{P}^n_{R}\to \mathbb{P}^n_{R}\times \mathbb{P}^n_{R}$ is affine, we know that the inverse image $\delta^{-1}(U)$ of any affine open $U\in \mathbb{P}^n_{R}$ is affine and therefore quasi-compact. Then, since the inverse image of any affine open under $\delta$ is quasi-compact, the diagonal $\delta$ is also quasi-compact. Hence $\mathbb{P}^n_{R}$ is quasi-separated.

    So, $\mathbb{P}^n_{R}$ is a geometric stack in the sense of \cite{SAG}, meaning it also satisfies the conditions of Definition \ref{ncgeo}, i.e. it is a non-connective geometric spectral stack \cite[Variant 1.3.3]{rok1}.
\end{proof}

We will show that in certain cases, quasi-coherent sheaves on (flat) projective space over a connective, Noetherian $\mathbb{E}_\infty$-ring $R$ admit an algebraic model. In order to see this, we will require a few lemmas.

\begin{lem}\label{dim}
    Let $R$ be a Noetherian ring of finite Krull dimension $m$, and let $\mathbb{P}^n_R$ be the projective space of dimension $n$ over $R$. Then the abelian category $\QC(\mathbb{P}^n_R)$ has finite global dimension $d\leq 2(m+n)$. 
\end{lem}

\begin{proof}
    First, we note that by Grothendieck vanishing \cite[III.2.7]{hart}, we have \[ H^p(\mathbb{P}^n_{R_*}; F)=0 \] for all $p>m+n$ and all $F_1\in \QC(\mathbb{P}^n_{R_*})$. As in \cite[II.7.3.3]{gode}, we have the local-to-global $\mathrm{Ext}$-spectral sequence \begin{equation} E_2^{p,q} = H^p(\mathbb{P}^n_{R_*}; \mathcal{E}xt_{\mathscr O}^q(F',F''))\Rightarrow \mathrm{Ext}_{\mathscr O}^{p
    +q}(F',F''). \label{extss}\end{equation} To see that $\mathcal{E}xt_{\mathscr O}^q(F',F'')=0$ for $q>m+n$, it suffices to check this condition on stalks, where $\mathcal{E}xt_{\mathscr O}^i(F',F'')_x=\mathrm{Ext}^i_{\mathscr O_x}(F'_x,F''_x)$; the latter is $\mathrm{Ext}$ groups of modules over a ring of dimension $\leq m+n$ (as the dimension of the stalks is bounded above by the dimension of the space). So $\mathcal{E}xt_{\mathscr O}^i(F',F'')_x=0$ for $i>m+n$, meaning $\mathcal{E}xt_{\mathscr O}^q(F',F'')=0$ for $q>m+n$. So, applying these bounds to our spectral sequence \ref{extss}, we see that $\mathrm{Ext}_{\mathscr O}^{p
    +q}(F',F'')=0$ whenever $p+q>2(m+n)$, proving the claim.
\end{proof}

\begin{lem}\label{abequiv}
    Let $R$ be a connective, Noetherian $\mathbb{E}_\infty$-ring such that $R_*$ has finite Krull dimension, and let $\Gamma$ be the coalgebra corresponding to the geometric stack $\mathbb{P}^n_R$. Then we have an equivalence of categories \[\QC(\mathbb{P}^n_{R_*})\simeq \comod_{\Gamma_*}.\]
\end{lem}

\begin{proof}
First, note that by \cite[5.4.1.7]{SAG}, we have a faithfully flat affine cover \[\coprod_{0\leq i\leq n}\Spet R[M_{\{i\}}]\to \mathbb{P}^n_R,\] along with an analogous cover \[\coprod_{0\leq i\leq n}\Spet R_*[M_{\{i\}}]\to \mathbb{P}^n_{R_*}.\] The coalgebras associated to these covers are given by the global sections of \[\coprod_{0\leq i\leq n}\Spet R[M_{\{i\}}]\times_{\mathbb{P}^n_R}\coprod_{0\leq j\leq n}\Spet R[M_{\{j\}}]\simeq \coprod_{0\leq i,j\leq n}\Spet R[M_{\{i,j\}}]\] and \begin{equation}\label{gamstar}\coprod_{0\leq i\leq n}\Spet R_*[M_{\{i\}}]\times_{\mathbb{P}^n_{R_*}}\coprod_{0\leq j\leq n}\Spet R_*[M_{\{j\}}]\simeq \coprod_{0\leq i,j\leq n}\Spet R_*[M_{\{i,j\}}]\end{equation} respectively. So, we can see that $\Gamma\simeq \prod_{0\leq i,j\leq n} R[M_{\{i,j\}}]$; it will suffice to show that $\Gamma_*$ agrees with the Hopf algebroid \ref{gamstar}. In order to see this, note that homotopy groups of ring spectra are computed in spectra, where finite products and coproducts agree, meaning they commute with finite products. So, we have \[\Gamma_*\cong \prod_{0\leq i,j\leq n} R_*[M_{\{i,j\}}]\] as needed.
\end{proof}

\begin{lem}\label{deg}
    Let $R$ be a connective $\mathbb{E}_\infty$-ring, and assume $R_*$ is concentrated in degrees divisible by $q+1$. Then $\Gamma_*\cong \prod_{0\leq i,j\leq n} R_*[M_{\{i,j\}}]$ is concentrated in degrees divisible by $q+1$.
\end{lem}
\begin{proof}
    We recall the computation of $\pi_0R[M_{\{i,j\}}]$ \cite[Remark 5.4.1.4]{SAG}. The monomorphism of monoids \[M_{\{i,j\}}\hookrightarrow \mathbb{Z}^{n+1}\] induces a monomorphism of graded rings \[\pi_0R[M_{\{i,j\}}]\hookrightarrow \pi_0R[x_0^{\pm 1}, \dots, x_n^{\pm 1}]\] into a Laurent polynomial algebra. This monomorphism induces an isomorphism onto its image \[\pi_0R[M_{\{i,j\}}]\xrightarrow{\cong} \pi_0R[\tfrac{x_0}{x_i}, \dots, \tfrac{x_n}{x_i},\tfrac{x_0}{x_j}, \dots, \tfrac{x_n}{x_j}],\] the subalgebra of the Laurent polynomial algebra generated by $\tfrac{x_k}{x_i},\tfrac{x_k}{x_j}$ for all $0\leq k\leq n$. Since each of these generators has degree $1+(-1)=0$, we have that each $\pi_0R[M_{\{i,j\}}]$ is concentrated in degree $0$. Then, since $\pi_0$ commutes with finite products, we have that $\Gamma_0=\pi_0(\prod_{0\leq i,j\leq n} R[M_{\{i,j\}}])$ is concentrated in degree 0. Finally, since $\Gamma$ is flat over $R$, we have that \[\Gamma_*\cong R_*\otimes_{R_0} \Gamma_0.\] Each object on the left hand side is concentrated in degrees divisible by $q+1$, meaning $\Gamma_*$ is also concentrated in degrees divisible by $q+1$.
\end{proof}

We now prove the main theorem of this section: 
\begin{thm}
    Let $R$ be a connective, Noetherian $\mathbb{E}_\infty$-ring  and let $\mathbb{P}^n_R$ be the flat projective space of dimension $n$ over $R$. Assume that $R_*$ has finite Krull dimension $m$ and is concentrated in degrees divisible by $2(m+n)+1$. Then $\mathbb{P}^n_R$ satisfies the conditions of Theorem \ref{main}, giving us a canonical equivalence \[h_{2(m+n)+1-d}\QC(\mathbb{P}^n_R)\simeq h_{2(m+n)+1-d}D^{\mathrm{per}}(\comod_{\Gamma_*}), \]
    with $\Gamma$ defined as in the previous lemma, where $d$ is the cohomological dimension of $\QC(\mathbb{P}^n_{R_*})\simeq \comod_{\Gamma_*}$.
\end{thm}
\begin{proof}
    By Lemma \ref{dim} and Lemma \ref{abequiv}, we have that the cohomological dimension of $\comod_{\Gamma_*}\simeq \QC(\mathbb{P}^n_{R_*})$ is $d\leq q=2(m+n)$, satisfying the second condition of Theorem \ref{main}; the first condition, that $\Gamma_*$ is concentrated in degrees divisible by $q+1=2(m+n)+1$, is Lemma \ref{deg}. So, since both conditions of Theorem \ref{main} are satisfied, we have an algebraic model \[h_{2(m+n)+1-d}\QC(\mathbb{P}^n_R)\simeq h_{2(m+n)+1-d}D^{\mathrm{per}}(\comod_{\Gamma_*}). \]
\end{proof}



\subsection{Chromatic algebraicity}\label{chrom}

In this section, we give a new proof of chromatic algebraicity (Theorem 2.4 in \cite{pst1}, 
recalled below) that is more algebro-geometric in nature and does not rely directly on the theory of synthetic spectra. In \cite{pp21} the authors rewrite the proof of the theorem in the language of adapted homology theories, applying the algebraicity theorem to the homology theory from Example \ref{chex} (taking $R=E(n)$). For the sake of completeness, we recall a few points prior to presenting our proof.

\begin{df}
    Recall that the category of $E_n$-local spectra $L_{E_n}\mathrm{Sp}$ for $E_n$ any $p$-local, Landweber exact homology theory of height $n$ depends only on the choice of the height $n$ and prime $p$. We denote this category, which we call the \textit{chromatic localization of spectra at height $n$ and the prime $p$}, by $L_n\mathrm{Sp}:=L_{E_n}\mathrm{Sp}$. For example, $E_n$ could be the $p$-local Lubin-Tate theory of height $n$, or the $p$-local Johnson-Wilson theory of height $n$, $E(n)$.
\end{df}

We also recall the following construction analogous to constructions from stable homotopy theory and classical algebraic geometry:
\begin{df}
For a nonconnective geometric spectral stack $\mathfrak{X}$, denote by 
\[\mathfrak{X}_{(p)}:= \mathfrak{X}\times_{\Spec S} \Spec S_{(p)}\] the $(p)$-\textit{localization} of $\mathfrak{X}$. 
\end{df}

In \cite[Theorem 2.3.1]{rok1}, Gregoric constructs a nonconnective geometric spectral stack known as the \textit{moduli stack of oriented formal groups}, which we denote by $\mathcal M^{\mathrm{or}}_{\mathrm{FG}}$. While the full definition of the stack is outside the scope of this paper, we recall the following theorem \cite[Theorem 2(iii)]{rok2}:

\begin{thm}
        There exists a canonical filtration by open nonconnective spectral substacks on the $(p)$ - localization of the moduli stack of oriented formal groups \[\mathcal M^{\mathrm{or},\leq 1}_{\mathrm{FG},(p)}\subseteq \cdots \subseteq \mathcal M^{\mathrm{or},\leq n}_{\mathrm{FG},(p)} \subseteq \cdots \subseteq \mathcal M^{\mathrm{or}}_{\mathrm{FG},(p)} \] such that for every finite height $1\leq n \leq \infty$, the functor $X\mapsto X\otimes \mathscr{O}_{\mathcal M^{\mathrm{or}}_{\mathrm{FG},(p)}}$ induces an equivalence of $\infty$-categories \[L_n\mathrm{Sp} \simeq \QC(\mathcal M^{\mathrm{or},\leq n}_{\mathrm{FG},(p)}) \] between the height $n$ chromatic localization of spectra at the prime $p$ and quasi-coherent sheaves on $\mathcal M^{\mathrm{or},\leq n}_{\mathrm{FG},(p)}$.
\end{thm}

We now give our proof of the chromatic algebraicity theorem of Pstragowski (and Patchkoria-Pstragowski, \cite[8.3.1]{pp21}):
\begin{thm}[Chromatic algebraicity]
    Let $E(n)$ be the Johnson-Wilson theory of height $n$ at a fixed prime $p$, and assume $2p-2>n^2+n$. Then, there exists a canonical equivalence \[ h_{2p-2-n^2-n}L_n\mathrm{Sp} \simeq h_{2p-2-n^2-n}D^{\mathrm{per}}(\mathrm{Comod}_{E(n)_*E(n)}) \] between the homotopy $(2p-2-n^2-n)$-categories of $E(n)$-local spectra and the periodic derived category of $E(n)_*E(n)$-comodules.
\end{thm}

\begin{proof}
By Theorem 2 in \cite{rok2}, we have an equivalence \[L_n\mathrm{Sp} \simeq \QC(\mathcal M^{\mathrm{or},\leq n}_{\mathrm{FG},(p)}) \] between the category of $E_n$-local spectra (for $E_n$ any Landweber exact homology theory of height $n$) and the category of quasi-coherent sheaves on the restriction of the moduli stack of oriented formal groups to those formal groups of height $\leq n$. Taking $E_n=E(n)$, the Johnson-Wilson theory of height $n$, when $2p-2>n^2+n$, we see that the conditions of Theorem \ref{main} are satisfied, as $E(n)$ is concentrated in degrees divisible by $q+1=2p-2$, and the category of $E(n)_*E(n)$-comodules is of homological dimension $d=n^2+n$ \cite[Theorem 2.4]{pst1}. So, we get a canonical equivalence \[ h_{2p-2-n^2-n}L_n\mathrm{Sp} \simeq h_{2p-2-n^2-n}D^{\mathrm{per}}(\mathrm{Comod}_{E(n)_*E(n)}) \] between the homotopy $(q+1-d)$-categories of $E(n)$-local spectra and the periodic derived category of $E(n)_*E(n)$-comodules.
\end{proof}






\bibliographystyle{abbrv}
\bibliography{algebraicmodelssag}{}

\end{document}